\documentclass{amsart}

\usepackage{amssymb}

\usepackage{epsfig}

\newtheorem{thm}{Theorem}%[section]

\newtheorem{cor}[thm]{Corollary}

\newtheorem{prop}[thm]{Proposition}

\theoremstyle{definition}
\newtheorem{defn}[thm]{Definition}

\newtheorem{say}[thm]{}
\newtheorem{exmp}[thm]{Example}

\newtheorem{rem}[thm]{Remark}          

\newtheorem{ack}{Acknowledgments}

\newtheorem{defn-thm}[thm]{Definition--Theorem}  %!!!!!!!!!!!!!!!!!!!!!!!!
\newtheorem{defn-lem}[thm]{Definition--Lemma}  %!!!!!!!!!!!!!!!!!!!!!!!!
\theoremstyle{remark}

\renewcommand{\c}[0]{{\mathbb C}}

\newcommand{\z}[0]{{\mathbb Z}}

\renewcommand{\r}[0]{{\mathbb R}}

\newcommand{\p}[0]{{\mathbb P}}

\newcommand{\map}[0]{\dasharrow}
\newcommand{\qtq}[1]{\quad\mbox{#1}\quad}

\newcommand{\pic}[0]{\operatorname{Pic}}

\newcommand{\aut}[0]{\operatorname{Aut}}

\newcommand{\diff}[0]{\operatorname{Diff}}

\newcommand{\bir}[0]{\operatorname{Bir}}

\newcommand{\onto}[0]{\twoheadrightarrow}

\begin{document}
\bibliographystyle{amsalpha}

\title[Cremona transformations  and 
 diffeomorphisms of  surfaces]{Cremona transformations \\ and 
\\ diffeomorphisms of  surfaces}
\author{J\'anos Koll\'ar \and Fr\'ed\'eric Mangolte}
\address{Princeton University, Princeton NJ 08544-1000}
\email{kollar@math.princeton.edu}
\address{Universit\'e de Savoie,  73376 Le Bourget du Lac Cedex, France.}
\email{mangolte@univ-savoie.fr}

\today

\maketitle

The simplest {\it Cremona transformation} of projective 3-space
is the involution
$$
\sigma \colon (x_0:x_1:x_2:x_3)\mapsto 
\Bigl(\frac1{x_0}:\frac1{x_1}:\frac1{x_2}:\frac1{x_3}\Bigr),
$$
which is a diffeomorphism outside the  tetrahedron $(x_0x_1x_2x_3=0)$.
More generally, if $L_i:=\sum_j a_{ji}x_j$ are linear forms defining
the faces of a tetrahedron, we get
the Cremona transformation
$$
\sigma_{\mathbf L} \colon (x_0:x_1:x_2:x_3)\mapsto 
\Bigl(\frac1{L_0}:\frac1{L_1}:\frac1{L_2}:\frac1{L_3}\Bigr)
\cdot \bigl(a_{ij}\bigr)^{-1},
$$
which is a diffeomorphism  outside the  tetrahedron $(L_0L_1L_2L_3=0)$.
 The vertices of the tetrahedron are called the {\it base points}.
If $Q$ is a quadric surface in $\p^3$, its image under a
Cremona transformation is, in general, a sextic surface.
However, if $Q$ passes through the 4 base points, then
 its image  $\sigma_{\mathbf L}(Q)$ is again a quadric surface in $\p^3$
passing through the 4 base points. In many cases,
we can view $\sigma_{\mathbf L}$ as a map of
$Q$ to itself.

The aim of this paper
is to show that these  Cremona transformations  
generate both the group of automorphisms and the group of
 diffeomorphisms of the sphere, the torus and of 
all non-orientable surfaces. 

Let us start with the sphere
$S^2:=(x^2+y^2+z^2=1)\subset \r^3$ and view this as 
the set of real points of the 
quadric $Q:=(x^2+y^2+z^2=t^2)\subset \p^3$ in projective 3-space.
Pick 2 conjugate point pairs $p,\bar p, q,\bar q$ on the complex
quadric $Q(\c)$ and let $\sigma_{p,q}$ denote the
Cremona transformation with base points $p,\bar p, q,\bar q$.
As noted above, $\sigma_{p,q}(Q)$ is another quadric surface. 
 The faces of the  tetrahedron determined
by these 4 points are disjoint from
$S^2$, hence  $\sigma_{p,q}$ is a diffeomorphism from $S^2$
to the real part of $\sigma_{p,q}(Q)$.
Thus $Q$ and $\sigma_{p,q}(Q)$ are projectively equivalent and
the corresponding Cremona transformation
$\sigma_{p,q}$ can be viewed as a diffeomorphism of $S^2$ to itself, 
 well defined up to left and right multiplication by $O(3,1)$.
It is also convenient to allow the points $p,q$ to coincide;
see (\ref{cubic.inv.exmp}) for a precise definition.
Let us call these the Cremona transformations
with  {\it imaginary base points}.
Our first result is that, algebraically, these generate
the automorphism group.

\begin{thm}\label{i.S2.thm} The  Cremona transformations with
imaginary base points
$\sigma_{p,q}$ and $O(3,1)$ generate 
$\aut(S^2)$.
\end{thm}

Most diffeomorphisms of $S^2$ are not algebraic, so the best
one can hope for is that these Cremona transformations
generate  $\diff(S^2)$ as topological group.
Equivalently, that $\aut(S^2)$ is a dense
subgroup of $\diff(S^2)$.
The results of \cite{luk-sib, luk-zam},
pointed out to us by  M.\ Zaidenberg, imply that 
 the group of  algebraic automorphisms is
 dense in the group of diffeomorphisms
for the sphere and the torus.
His methods, reviewed in (\ref{luk.say}),
use the $SO(3,\r)$ action on the sphere
and the torus action on itself.

In order to go futher, first we need to deal with
diffeomorphisms with fixed points.
Building on \cite{bh}, it is proved in \cite{hm1}  that $\aut(S^2)$
is $n$-transitive for any $n\geq 1$.
 Using this, it is easy to see (\ref{pf.of.i.S2n.thm})
that   the  density property also holds with
assigned fixed points.

\begin{cor}\label{i.S2n.thm} 
$\aut(S^2, p_1,\dots, p_n)$ is dense in $ \diff(S^2, p_1,\dots, p_n)$
for any finite set of distinct points
$p_1,\dots, p_n\in S^2$, 
where $\aut(\ )$ denotes
the group of algebraic
automorphisms of $S^2$ 
 fixing $p_1,\dots, p_n$
and  $ \diff(\ )$ the  group of  diffeomorphisms  fixing $p_1,\dots, p_n$.
\end{cor}

Note that, for  a real algebraic variety $X$, the semigroup of algebraic
diffeomorphisms is usually much bigger than the
group of algebraic
automorphisms $\aut(X)$. For instance, 
$x\mapsto x+\frac1{x^2+1}$ is an algebraic
diffeomorphism of $\r$ (and also of $\r\p^1\sim S^1$),
but its inverse involves square and cube roots.
The  difference  is best seen
in the case of the circle $S^1=(x^2+y^2=1)$.

Essentially by the Weierstrass approximation theorem, any
differentiable map $\phi \colon S^1\to S^1$ can be approximated by
polynomial maps $ \Phi \colon S^1\to S^1$. By contrast, the
group of algebraic automorphisms of $S^1$ is the real  orthogonal group
$O(2,1)\cong PGL(2,\r)$, which has real dimension 3. 
Thus $\aut(S^1)$ is   a very small closed 
subgroup in the infinite dimensional group
 $\diff(S^1)$.

The Cremona transformations with real base points
do not give diffeomorphisms of $S^2$;
they are not even defined at the real base points.
Instead, they
give generators of the mapping class groups of
non-orientable surfaces.

Let $R_g$ be a non-orientable, compact surface of genus $g$ without boundary.
  Coming from algebraic geometry,
we prefer to think of it as $S^2$ blown up at $g$ points
$p_1,\dots, p_g\in S^2$.
Topologically,  $R_g$ is obtained from $S^2$ 
by replacing $g$ discs centered at the $p_i$  
by $g$ M\"obius bands.
Up to isotopy,  a blow-up form of $R_g$ is equivalent to
giving $g$ disjoint embedded M\"obius bands 
$M_1,\dots, M_g\subset R_g$.  

There are two ways to think of a 
Cremona transformation with real base points as giving
elements of the mapping class group of $R_g$.

Let us start with the case when there are four real base points
$p_1,\dots,p_4$.
We can factor the Cremona transformation $\sigma_{p_1,\dots,p_4}$
 as
$$
\sigma_{p_1,\dots,p_4} \colon
Q \stackrel{\pi_1}{\leftarrow} B_{p_1,\dots,p_4}Q \stackrel{\pi_2}{\to} Q
$$
where on the left $\pi_1 \colon B_{p_1,\dots,p_4}Q\to Q$ is the blow up 
of $Q$ at
the 4 points $p_1,\dots,p_4$ and on the right
$\pi_2 \colon B_{p_1,\dots,p_4}Q\to Q$
contracts the birational transforms of the
circles $Q\cap L_i$ where the $\{L_i\}$ are the faces of the
tetrahedron with vertices $\{p_i\}$.
In Figure~\ref{fig.cremona}, the $\bullet$ represent the 4 base points.
On the left hand side, the 4 exceptional curves $E_i$
lie over the four points marked $\bullet$.
On the right hand side, the images of the
$E_i$ are 4 circles, each passing through 3 of the 4 base points.
Since $\sigma_{p_1,\dots,p_4}$ is an involution, dually,
the four points marked $\bullet$ on the right hand side
  map to the 4 circles on the left hand side.
   \begin{figure}[ht]
   \epsfysize=4cm
    \epsfbox{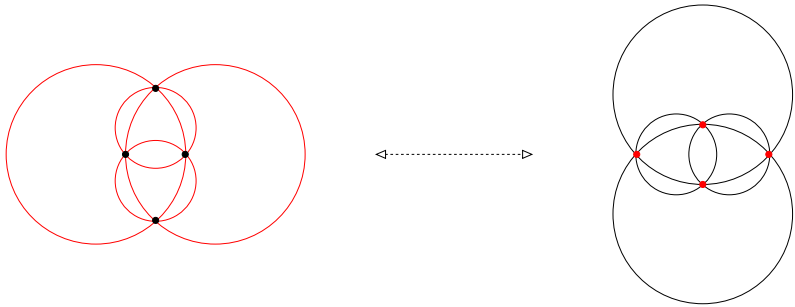} 
 \caption{Cremona transformation with four real base points.}
        \label{fig.cremona}
      \end{figure}

A Cremona transformation $\sigma_{p_1,p_2, q,\bar q}$  with
 2 real and a conjugate complex pair
of base points  act similarly.  Here only
two  M\"obius bands are altered.

In general, 
 we can think of the above real Cremona transformation
$\sigma_{p_1,\dots,p_4}$ as a topological operation that
 replaces the set of $g$ M\"obius bands
$\bigl(M_1,\dots, M_g\bigr)$
by a new set 
$\bigl(M'_1,\dots, M'_4, M_5,\dots,  M_g\bigr)$.
In this version, $\sigma_{p_1,\dots,p_4}$ is the identity
on the surfaces but acts nontrivially on the
set of isotopy classes of $g$ disjoint M\"obius bands.
One version of our result says that the
transformations  $\sigma_{p_1,\dots,p_4}$ 
and $\sigma_{p_1,p_2, q,\bar q}$  act transitively on the
set of isotopy classes of $g$ disjoint M\"obius bands.

The other way to view  $\sigma_{p_1,\dots,p_4}$
is as follows. First, we obtain an isomorphism 
$$
\sigma'_{p_1,\dots,p_4} \colon B_{p_1,\dots,p_g}S^2\cong 
B_{q_1,\dots,q_g}S^2
$$
 for some $q_1,\dots,q_g\in S^2$.
Under this isomorphism, the exceptional curve
$E(p_i)\subset B_{p_1,\dots,p_g}S^2$ is mapped to the exceptional curve
$E(q_i)\subset B_{q_1,\dots,q_g}S^2$ for $i\geq 5$ 
and to the circle passing through the points
$\{q_j: 1\leq j\leq 4, j\neq i\}$ for $i\leq 4$.
As we noted above, there is an automorphism
$\Phi\in \aut(S^2)$ such that
$\Phi(q_i)=p_i$ for $1\leq i\leq n$.
Thus 
$$
\Phi\circ \sigma'_{p_1,\dots,p_4} \colon B_{p_1,\dots,p_g}S^2
\stackrel{\cong}{\longrightarrow}
B_{p_1,\dots,p_g}S^2
$$
is an automorphism of $B_{p_1,\dots,p_g}S^2$
which maps $E(p_i)$ to 
$E(q_i)$ for $i\geq 5$ 
and to a simple closed curve passing through the points
$\{p_j: 1\leq j\leq 4, j\neq i\}$ for $i\leq 4$.

\begin{thm}\label{i.mpc.thm} For any $g$, 
 the Cremona transformations with $4,2$ or $0$ real base points
 generate the  (non-orientable) mapping class group 
${\mathcal M}_{g}$.
\end{thm}

Finally, we can put these results together to obtain a general 
approximation theorem for 
diffeomorphisms of such real algebraic surfaces.

\begin{thm} \label{homeo.gen.dense}
 Let $R$ 
be a compact, smooth, real algebraic surface birational to $\p^2$
and $q_1,\dots, q_n\in R$ distinct marked points. Then
the group of algebraic automorphisms
$\aut(R, q_1,\dots, q_n)$ is dense in 
 $\diff(R, q_1,\dots, q_n)$.

As a topological manifold,  here $R$ can be
the sphere, the torus or any non-orientable  surface
$\r\p^2\#\cdots\# \r\p^2$. 
\end{thm}

\begin{say}[Other algebraic varieties]
Similar assertions definitely fail for most other
 algebraic varieties. Real algebraic varieties 
of general type have
only finitely many birational automorphisms.
(See \cite{ueno75} for an introduction to these questions.)
For varieties whose Kodaira dimension is between 0 and the dimension,
every birational automorphism preserves the Iitaka fibration.
If the Kodaira dimension is 0
(e.g., Calabi-Yau varieties, Abelian varieties),
then every  birational automorphism preserves the canonical class,
that is, a volume form, up to sign. The
automorphism group is finite dimensional but may have
infinitely many connected components.
In particular, using \cite{comessatti}, for surfaces we obtain
the following.

\begin{prop} Let $S$ be a smooth real algebraic surface.
If $S(\r)$ is an orientable surface of genus $\geq 2$
then $\aut(S)$ is not dense in $\diff\bigl(S(\r)\bigr)$.\qed
\end{prop}

If $X$ has Kodaira dimension $-\infty$,
then  every  birational automorphism preserves the MRC fibration
\cite[Sec.IV.5]{rc-book}.
Thus the main case when density could hold is when the variety
is rationally connected \cite[Sec.IV.3]{rc-book}.
It is clear that the analog of
(\ref{i.S2.thm}) fails even 
for most geometrically rational real algebraic surfaces.
Consider, for instance, the case when
$R\to \p^1$ is a minimal conic bundle with at least 8 singular fibers.
Then $\aut(R)$ is infinite dimensional, but
every automorphism of $R$ preserves the
conic bundle structure \cite[Thm.~1.6(iii)]{is}.  
Conic bundles with 4 singular fibers are probably the
only other case where the analog of (\ref{homeo.gen.dense})
holds.

The results of \cite{luk-sib} imply that
 $\aut(S^n)$ is  dense in $\diff(S^n)$ for every $n\geq 2$
and, similarly,   $\aut(T^n)$ is  dense in $\diff(T^n)$ for every $n\geq 2$,
where $T^n$ is the $n$-dimensional torus.
It is not clear to us what happens with other 
varieties birational to $\p^n$.
\end{say}

\begin{say}[History of related questions]
There are many results in real algebraic geometry
that endow certain topological spaces with a real 
algebraic structure or approximate smooth maps by 
real algebraic morphisms.
 In particular real rational models of surfaces were studied in
\cite{bcr}, \cite{ma06}
and approximations of smooth maps to  spheres by 
real algebraic morphisms  were investigated in
\cite{bk1,bk2}, 
\cite{bks}, 
\cite{Ku}, \cite{jk}, \cite{jm}, 
\cite{ma06}.

An indication that $\aut(S^2)$ is surprisingly large
comes from \cite{bh}, with a more precise version developed in 
\cite{hm1}. 
\end{say}

\begin{say}[Plan of the proofs] 
In Section \ref{sec.gens} we  prove that the
Cremona transformations with imaginary base points
generate $\aut(S^2)$. 
Next, in Section~\ref{sec.identity}, we prove (\ref{homeo.gen.dense}) for the
identity components.  If 
$\phi \colon R\to R$ is homotopic to the identity, then
$\phi$ can be written as the composite of
diffeomorphisms $\phi_i \colon R\to R$ such that each
$\phi_i$ is the identity outside a small open set
$W_i\subset R$. Moreover, we can choose the $W_i$ in such a way
that for every $i$ there is a morphism $\pi_i \colon R\to S^2$ that is
an isomorphism on $W_i$.  The map $\phi_i$ then pushes down to
a diffeomorphism of $S^2$. We take an approximation there and 
lift it to $R$. 

 The  case   $R=S^1\times S^1$ follows from \cite{luk-zam}.

Generators of the mapping class group of non-orientable surfaces
have been
written down by \cite{chi} and \cite{kor}. 
In Section \ref{sec.mapping} we describe
 a somewhat  different set
of generators. We thank  M.\ Korkmaz for his help in proving
these results.

Theorem \ref{i.mpc.thm} is proved in Section \ref{mapping.cl.proofs}.
 We show by explicit
constructions that our generators are given by Cremona transformations.

\end{say}

\begin{ack} We thank W.\ Browder, N.\ A'Campo, S.\ Cantat, D.\ Gabai, 
 J.\ Huisman, I.\ Itenberg, V.\ Kharlamov and A.\ Okounkov for many useful 
conversations.
 We are especially grateful to M.\ Korkmaz
for his help with understanding the non-orientable mapping class group.
M.\ Zaidenberg called our attention to the papers of Lukacki{\u\i}.
These enabled us to shorten the proofs and to improve the results.

Partial financial support  for JK
was provided by  the NSF under grant number 
DMS-0500198. The research of FM was partially supported by the  
ANR grant ``JCLAMA'' of the french ``Agence Nationale de la Recherche.'' 
He  also benefited from the hospitality of Princeton University.
\end{ack}

\section{Generators of $\aut(S^2)$}\label{sec.gens}

Max Noether proved that the involution
$(x,y,z)\mapsto \bigl(\frac1{x},\frac1{y},\frac1{z}\bigr)$ and
$PGL_3$ generate the group of birational self-maps $\bir(\p^2)$ over $\c$.
Using similar ideas,
 \cite{rv}  proved that $\aut(\p^2(\r))$ is generated by linear automorphisms
 and certain real algebraic automorphisms of degree 5. 
In this section, we prove that $\aut(S^2)$ is generated by linear automorphisms
 and by the $\sigma_{p.q}$. The latter are
 real algebraic automorphisms of degree 3. 

\begin{exmp}[Cubic involutions of $\p^3$]\label{cubic.inv.exmp}
 On $\p^3$ take coordinates $(x,y,z,t)$.
We need two types of cubic involutions of $\p^3$.
Let us start with the  Cremona
transformation 
$$
(x,y,z,t)\mapsto \bigl(\tfrac1{x},\tfrac1{y},\tfrac1{z},\tfrac1{t}\bigr)=
\tfrac1{xyzt}(yzt, ztx, txy, xyz)
$$
whose base points are the 4 ``coordinate vertices''.
We will need to put the base points at complex conjugate
point pairs, say $(1,\pm i,0,0), (0,0,1,\pm i)$.
Then the above involution becomes
$$
\tau \colon 
(x,y,z,t)\mapsto \bigl((x^2+y^2)z, (x^2+y^2)t, (z^2+t^2)x,(z^2+t^2)y\bigr).
$$
Check that 
$$
\tau^2(x,y,z,t)=(x^2+y^2)^2(z^2+t^2)^2\cdot (x,y,z,t), 
$$
thus $\tau$ is indeed a rational  involution on $\p^3$.

Consider a general quadric  passing through
the points
$(1,\pm i,0,0), (0,0,1,\pm i)$. It is 
of the form
$$
Q=Q_{abcdef}(x,y,z,t):=a(x^2+y^2)+b(z^2+t^2)+cxz+dyt+ext+fyz.
$$
By direct computation, 
$$
Q_{abcdef}\bigl(\tau(x,y,z,t)\bigr)=(x^2+y^2)(z^2+t^2)\cdot Q_{abcdfe}(x,y,z,t).
$$
(Note that $ef$ changes to $fe$.
Thus, if $e=f$, then  $\tau$ restricts to an involution of the quadric
$(Q=0)$ but not in general.)

Assume now that we are over $\r$. We claim that $\tau$ is regular on
the real points if $a,b\neq 0$.
The only possible problem is with points where 
$(x^2+y^2)(z^2+t^2)=0$. Assume that $(x^2+y^2)=0$.
Then $x=y=0$ and so $Q(x,y,z,t)=0$ gives that
$b(z^2+t^2)=0$ hence $z=t=0$, a contradiction.

Whenever $Q$ has signature $(3,1)$, we can view
$(Q=0)$ as a sphere and then $\tau$  gives a  real algebraic
automorphism of the sphere $S^2$, which
 is well defined  up to left and right multiplication by $O(3,1)$.
A priori the automorphisms depend on $a,b,c,d,e,f$,
so let us denote them by $\tau_{abcdef}$.

Given $S^2$, the above $\tau_{abcdef}$
depends on the choice
of the base points, that is, 2 conjugate pairs of
points on the complex quadric $S^2(\c)$.
The group $O(3,1)$ has real dimension 6. Picking
2 complex points has real dimension 8. So 
the $\tau_{abcdef}$ 
should give a real  2-dimensional family
of automorphisms modulo $O(3,1)$.

We also need a degenerate version of the Cremona transformation
when the 4 base points come together to a pair of points.
With  base points  $(1,0,0,0)$ and $(0,1,0,0)$, we get
$$
(x,y,z,t)\mapsto (xz^2, yt^2, zt^2, z^2t).
$$ 
If we put the base points at $(1,\pm i,0,0)$
then we get the transformation
$$
\sigma' \colon (x,y,z,t)\mapsto 
\bigl(y(z^2-t^2)+2xzt, x(t^2-z^2)+2yzt, t(z^2+t^2), z(z^2+t^2)\bigr).
$$

Take any quadric of the form 
$$
Q=Q'_{abcdef}(x,y,z,t):=a(x^2+y^2)+bz^2+czt+dt^2+e(xt+yz)+f(xz-yt).
$$
By direct computation, 
$$
Q'_{abcdef}\bigl(\sigma'(x,y,z,t)\bigr)=(z^2+t^2)^2\cdot 
Q'_{adcbef}(x,y,z,t).
$$
As before, if $Q'$ has signature $(3,1)$, we can view
$(Q'=0)$ as a sphere and then $\sigma'$ gives a  real algebraic
automorphism of the sphere $S^2$, which
 is well defined  up to left and right multiplication by $O(3,1)$.
Let us denote them by $\sigma_{abcdef}$.
Despite the dimension count, the group  $O(3,1)$ does not act
with a dense orbit on the set
of complex conjugate point pairs and complex conjugate directions.
Indeed, after complexification, the quadric becomes
$\p^1\times \p^1$ and  we can chose the two
points to be $p_1:=(0,0)$ and $p_2:=(\infty,\infty)$.
The subgroup fixing these two points is
$\c^*\times \c^*$ and the diagonal acts trivially on
the tangent directions at both of the points $p_i$.
Thus the  $\sigma_{abcdef}$ form a 1-dimensional family.
\end{exmp}

\begin{thm}\label{generators.thm}
 The group of algebraic automorphisms of $S^2$
is generated by $O(3,1)$, the $\tau_{abcdef}$ and $\sigma_{abcdef}$.
\end{thm}

\begin{rem} It is possible that the $\tau_{abcdef}$ alone generate
$\aut(S^2)$. In any case, as the 4 base points come together
to form 2 pairs, the  $\tau_{abcdef}$
converge to the corresponding $\sigma_{abcdef}$.
Thus the $\tau_{abcdef}$ generate a dense subgroup of $\aut(S^2)$
(in the $C^{\infty}$-topology.)

One reason to use the $\sigma_{abcdef}$ is that, 
as the proof  shows,  the 
$\tau_{abcdef}$ and $\sigma_{abcdef}$ together generate $\aut(S^2)$
in an ``effective manner.'' By this we mean the following.

Any rational map $\Phi \colon S^2\map S^2$ can be given by
4  polynomials
$$
\Phi(x,y,z,t)=\bigl(\Phi_1,\Phi_2, \Phi_3,\Phi_4\bigr).
$$
Note that $\Phi$ does not determine the $\Phi_i$
uniquely, but there is a ``minimal'' choice.
We can add any multiple of $x^2+y^2+z^2-t^2$ to the $\Phi_i$
and we can cancel common factors. 
We choose $\max_i\{\deg \Phi_i\}$ to be minimal and
call it the {\it degree} of $\Phi$. 
It is denoted by $\deg \Phi$. (It is easy to see that these 
minimal $\Phi_i$ are unique up to a multiplicative constant.)
Note that $\deg\Phi=1$ iff $\Phi\in O(3,1)$.

By ``effective'' generation we mean that
given any $\Phi\in \aut(S^2)$ with $\deg \Phi>1$, there is a
$\rho$ which is either of the form 
$\tau_{abcdef}$ or  $\sigma_{abcdef}$ such that
$$
\deg\bigl(\Phi\circ \rho)<\deg \Phi.
$$
\end{rem}

\begin{say}[Proof of (\ref{generators.thm})]

The proof is an application of the Noether-Fano method.
See \cite[Secs. 2.2--3]{ksc} for details.

  Let  $k$ be a field and $Q\subset \p^3$ a quadric 
defined over $k$. Assume that $\pic(Q)=\z[H]$
where $H$ is the hyperplane class.
Let  $Q'$ be any other quadric and $\Phi \colon Q\map Q'$ a birational map.
Then $\Gamma:=\Phi^*|H_{Q'}|$ is a 3-dimensional linear system on $Q$ and
$\Gamma\subset |dH_Q|$ for some $d$. Let $p_i$ be 
the (possibly infinitely near)  base points of $\Gamma$
(over $\bar k$)  and
$m_i$  their multiplicities. 
As in \cite[2.8]{ksc}, we have the equalities
$$
\Gamma^2-\sum m_i^2=\deg Q'\qtq{and}
\Gamma\cdot K_Q+\sum m_i = \deg K_{Q'}.
$$
In our case, these become
$$
\sum m_i^2=2d^2-2\qtq{and} \sum m_i=4d-4.
$$

Next we see how the transformations
$\tau_{abcdef}$ and $\sigma_{abcdef}$
change the degree of a linear system $\Gamma$.

\begin{exmp}[Cremona transformation on a quadric]

For the $\tau_{abcdef}$ series,
pick 4 distinct points $p_1,\dots, p_4\in Q$ 
such that no two are on a line in $Q$, not all 4 on a conic
and assume that
$s:=m_1+\cdots+m_4>2d$. Blow up the 4 points and contract
the 4 conics that pass through any 3 of them.
The $p_i$ are replaced by 4 base points of
multiplicities $2d-s+m_i$. Their sum is $8d-4s+s=8d-3s$.
Thus $4d-4=\sum m_i$ is replaced by
$\sum m_i-s+(8d-s)$,
hence $d$ becomes $d-(s-2d)<d$.

For  $\sigma_{abcdef}$,
pick 2 distinct points $p_1,p_2\in Q$
and 2 infinitely near points $p_3\to p_1$ and $p_4\to p_2$ 
such that no two are on a line in $Q$, not all 4 on a conic 
and assume that
$s:=m_1+\cdots+m_4>2d$. Blow up the  points 
$p_1,p_2$ and then the  points 
$p_3,p_4$. After this, we can
contract
the two conics that pass through $p_1+p_2+p_3$ (resp.\ $p_1+p_2+p_4$)
and we can also contract the birational transforms
of the exceptional curves over $p_1$ and $p_2$.
The rest of the computation is the same.
The $p_i$ are replaced by 4 base points of
multiplicities $2d-s+m_i$. Their sum is $8d-4s+s=8d-3s$.
Thus $4d-4=\sum m_i$ is replaced by
$\sum m_i-s+(8d-s)$
hence $d$ becomes $d-(s-2d)<d$.
\end{exmp}

Thus, as long as we can find $p_1,\dots, p_4\in Q$ 
(or infinitely near) such that
$m_1+\cdots+m_4>2d$, we can lower $\deg\Phi$ by a suitable
degree 3 Cremona transformation.

In order to find such $p_i$, assume 
first to the contrary that  $m_i\leq d/2$ for every $i$.
Then
$$
2d^2-2=\sum m_i^2\leq \tfrac{d}{2} \sum m_i=\tfrac{d}{2}(4d-4)=2d^2-2d.
$$
This is a contradiction, unless $d=1$ and $\Phi$ is a linear
isomorphism.

If we work over $\r$ and we assume that there are no real base points,
then we have at least one  complex conjugate pair of
base points with multiplicity $m_i>d/2$.
We are done if we have found 2 such pairs.

In any case, up to renumbering the points,
 we have $m_1=m_2=\frac{d}{2}+a$ for some $d/2\geq a>0$.
Assume next that all the other $m_j\leq \frac{d}{2}-a$.
Then 
$$
\begin{array}{rcl}
2d^2-2=\sum m_i^2 &\leq & 2\bigl(\tfrac{d}{2}-a\bigr)^2+
\bigl(\tfrac{d}{2}-a\bigr)\bigl(\sum m_i-d+2a\bigr)\\
&=&2\bigl(\tfrac{d}{2}-a\bigr)^2+
\bigl(\tfrac{d}{2}-a\bigr)\bigl(4d-4-d+2a\bigr).
\end{array}
$$
By expanding, this becomes
$$
(a+2)(d-4)\leq -6.
$$
Thus $d\in\{1,2,3\}$. If $d=3$ then $a+2\geq 6$ so $d/2\geq a\geq 4$
gives a contradiction. If $d=2$ then we get $a=1$. 
Thus $\Gamma$ consists of quadric sections with
singular points at $p_1,p_2$. These are necessarily reducible
(they have $p_a=1$  with 2 singular points), again impossible.

We also need to show that no two of the points lie
on a line and not all 4 are on a conic. 
For any line $L\subset Q(\c)$, 
$(L\cdot \Gamma)=d$ gives that 
$$
\sum_{i:p_i\in L}m_i\leq d.
$$
In particular, $m_i\leq d$ for every $i$ and
if $p_i,p_j$ are on a line then $m_i+m_j\leq d$.
Thus out of $p_1,\dots,p_4$ only $p_3,p_4$ could be on a line.
But $p_3,p_4$ are conjugates, thus they would be on a real
line. There is, however, no real line on $S^2$.

Similarly, for any conic  $C\subset Q(\c)$, 
$(C\cdot \Gamma)=2d$ gives that 
$\sum_{i:p_i\in C}m_i\leq 2d$.
Thus not all 4 points are on a conic. 
\end{say}

\begin{rem} Note that we started the proof over an arbitrary field,
but at the end we had to assume that that we worked over $\r$.
For a quadric surface $Q$ with Picard number one,
the above method should give generators for the group
$\bir^*(Q)$ of those birational self-maps that are regular 
along $Q(k)$. However, for other fields $k$, other
generators also appear if there are more than 2
conjugate base points.
\end{rem}

\section{The identity component}\label{sec.identity}

The purpose of this section is to prove (\ref{homeo.gen.dense}) for the
identity components. 
For the sphere and the torus these were done by Lukacki{\u\i}.
 Next we prove (\ref{homeo.gen.dense}) for the
identity components in the case $R$ is the non-orientable surface $R_g$.

\begin{say}[The results of Lukacki{\u\i}]\label{luk.say}
The paper \cite[Thm.~2]{luk-sib}  proves that $SO(n+1,1)$ is a maximal
closed subgroup of $\diff_0(S^n)$. In particular, 
$O(n+1,1)$ and anything else generate a dense subgroup of $\diff(S^n)$.

Since this result seems not to have been well known, let us
give a quick review of the steps of the proof.

We start with the Lie algebra of polynomial vector fields
$H^0(S^n,T_{S^n})$. Its structure   as an $\mathfrak{so}(n+1)$
representation   was described by \cite{kir},
including the highest weight vectors.

As we go from $\mathfrak{so}(n+1)$ to $\mathfrak {so}(n+1,1)$,
we get extra unipotent elements and their action on the
highest weight vectors can be computed explicitly.
One obtains that $\mathfrak {so}(n+1,1)$ is a 
maximal Lie subalgebra of $H^0(S^n,T_{S^n})$.
This implies that $SO(n+1,1)$ is a maximal
connected closed subgroup of $\diff_0(S^n)$.
It is easy to check that  $SO(n+1,1)$ is its own normalizer, which rules
out all disconnected subgroups as well.

The paper \cite{luk-zam} gives generators
of the Lie algebra $H^0(T^n,T_{T^n})$
 where $T^n$ denotes the $n$-dimensional torus
$$
T^n:=\bigl(x_1^2+y_1^2-1=\cdots=x_n^2+y_n^2-1=0\bigr)\subset \r^{2n}.
$$
 This is again through explicit Lie theory.
Up to coordinate changes by $GL(n,\z)$, the generators
 are the shears
$$
g(x_1,\dots,x_{n-1})\cdot 
\Bigl( \frac{\partial}{\partial x_n} -\frac{\partial}{\partial y_n}\Bigr)
\qtq{and}
y_n\cdot 
\Bigl(\frac{\partial}{\partial x_n} -\frac{\partial}{\partial y_n}\Bigr).
$$
(Using polar angles $\phi_i$, the latter
is the vector field $\sin\phi_n\cdot \bigl(\partial/\partial \phi_n\bigr)$.)
Up to a factor of 2, this is exactly the tangent vector field
corresponding to the unipotent group
$$
\Bigl(
\begin{array}{cc}
1 & t\\
0 & 1
\end{array}
\Bigr)\subset PGL(2,\r)\cong O(2,1)\qtq{acting on $S^1$.}
$$
\end{say}

\begin{defn}\label{twist.defn}
Let $X$~and $Y$ be real algebraic manifolds and let~$I$ be any subset of~$X$. 
A map~$f$ from~$I$ into $Y$ is \emph{algebraic} if there is a Zariski 
open subset~$U$
of~$X$ containing~$I$ such that~$f$ is the restriction of an
algebraic map from~$U$ into~$Y$.

Consider the standard sphere $S^2\subset \r^3$
and let $L$ be a line through the origin.
Choose coordinates such that $L$ is the $x$-axis
and 
$S^2:=(x^2+y^2+z^2=1)\subset \r^3$.
Let $M\colon[-1,1] \to O(2)$ be a real algebraic map. Then
$$
\Phi_M \colon S^2 \to S^2   
\qtq{given by}
(x,y,z)\mapsto \bigl(x,(y,z)\cdot M(x)\bigr)
$$
is an automorphism of $S^2$, called
the {\it twisting map} with axis $L$ and associated to $M$.
A conjugate of a twisting map by an element of $O(3,1)$
is also called a twisting map.
\end{defn}

The following results are proved in \cite{hm1}.

\begin{thm} \label{hm.recall.thm}
Notation as above.
\begin{enumerate}
\item Any $C^{\infty}$ map $M_0:[-1,1] \to O(2)$
can be approximated by  real algebraic maps
$M_s:[-1,1] \to O(2)$. Moreover, given finitely many points
$t_i\in [-1,1]$, we can choose the $M_s$
such that $M_s(t_i)=M_0(t_i)$ for every $i$.
\item Given distinct points $p_1,\dots, p_m$ and 
$q_1,\dots, q_m$ there are two twisting maps
(with different axes) $\Phi_1$ and $\Phi_2$
such that $\Phi_1\circ \Phi_2(q_i)=p_i$ for every $i$.
Moreover, 
\begin{enumerate}
\item if $p_j=q_j$ for some values of $j$ then we can assume that
$\Phi_1(q_j)=\Phi_2(q_j)=q_j$ for these values of $j$, and
\item if $p_i$ is near $q_i$ for every $i$ then
we can assume that the $\Phi_1, \Phi_2$ are near the identity.
\end{enumerate}
\item Let $R$ be any real algebraic surface that is obtained from
$S^2$ by repeatedly blowing up $m$ real  
(possibly infinitely near) points and let
$r_1,\dots, r_n$ be points in $R$. Then 
there are (nonunique) distinct points
 $p_1,\dots, p_m$ and 
$q_1,\dots, q_n$ and an isomorphism
$\phi:R\to B_{p_1,\dots, p_m}S^2$ such that $\phi(r_i)=q_i$.
\end{enumerate}
\end{thm}

By adding more points in (\ref{hm.recall.thm}.3) and compactness,
we obtain the following stronger version:

\begin{cor}\label{good.support.sets}
Let $R$ be any real algebraic surface that is obtained from
$S^2$ by repeatedly blowing up $m$ real  
(possibly infinitely near) points and let
$r_1,\dots, r_n$ be points in $R$. 
There is a finite open cover $R=\cup_j W_j$ 
such that for every $j$ 
there are  distinct points
 $p_{1j},\dots, p_{mj}, q_{1j},\dots, q_{nj}\in S^2$ and an isomorphism
$\phi_j:R\to B_{p_{1j},\dots, p_{mj}}S^2$ such that $\phi_j(r_i)=q_{ij}$
and $\phi_j(W_j)\subset S^2\setminus\{p_{1j},\dots, p_{mj}\}$.
\qed
\end{cor}

\begin{say}[Proof of (\ref{i.S2n.thm})] \label{pf.of.i.S2n.thm}
Let $p_1,\dots, p_n,q\in S^2$ be any finite set of distinct points, and 
let  $\phi \in \diff(S^2, p_1,\dots, p_n)$.
By (\ref{luk.say}) there are automorphisms 
$\psi_s \in \aut(S^2)$ such that 
$\psi_s$ converges to $\phi$.

For any $s$ and $i$, set $q_i^s:= \psi_s(p_i)$. 
As $\psi_s$ converges to $\phi$, the $q_i^s$ converge to $p_i$ for every $i$. 
By (\ref{hm.recall.thm}.2.b) there are
automorphisms $\Phi_s$ such that $\Phi_s(q_i^s)=p_i$
and  $\Phi_s$ converges to the identity.
Thus the composites $\Phi_s\circ \psi_s$ are in
$\aut(S^2, p_1,\dots, p_n)$ and they
converge to $\phi$.
\end{say}

\begin{prop}\label{homeo.identity.dense}
Let $R$ be any real algebraic surface that is obtained from
$S^2$ by repeatedly blowing up $g$ real  
(possibly infinitely near) points and let
$r_1,\dots, r_n$ be points in $R$. 
Then
the group $\aut_0(R, r_1,\dots, r_n)$ of 
algebraic automorphisms homotopic to identity
 is dense in  $\diff_0(R, r_1,\dots, r_n)$.
\end{prop}

\begin{proof} Let
$\phi \colon R\to R$ be a diffeomorphism fixing $r_1,\dots, r_n$, 
and homotopic to the identity. 
Choose $R=\cup_jW_j$ as in (\ref{good.support.sets}). 
By a partition of unity argument,
$\phi$ can be written as the composite of
diffeomorphisms $\phi_{\ell} \colon R\to R$ fixing $r_1,\dots, r_n$
 such that each
$\phi_{\ell}$ is the identity outside some
$W_j\subset R$. 

In particular, each $\phi_{\ell}$ descends to a
diffeomorphism $\phi'_{\ell}$ of $S^2$ which fixes the points
$p_{1j},\dots, p_{gj}$ and 
$q_{1j},\dots, q_{nj}$. By
(\ref{i.S2n.thm}), we can approximate
$\phi'_{\ell}$ by algebraic automorphisms 
$\Phi'_{\ell,s}$ fixing
all the points $p_{1j},\dots, p_{gj}$ and 
$q_{1j},\dots, q_{nj}$.
Since the $\Phi'_{\ell,s}$ fix $p_{1j},\dots, p_{gj}$,
they lift to 
algebraic automorphisms 
$\Phi_{\ell,s}$ of
$R\cong B_{p_{1j},\dots, p_{gj}}S^2$ fixing
 the points $r_1,\dots, r_n$.
The composite of the $\Phi_{\ell,s}$ then converges to $\phi$.
\end{proof}

\section{Generators of the mapping class group}\label{sec.mapping}

\begin{defn} Let $R$ be a compact, closed surface and
$p_1,\dots, p_n$ distinct points on $R$. 
The {\it mapping class group}
is the group of connected components of
those diffeomorphisms $\phi \colon R\to R$ such that $\phi(p_i)=p_i$
for $i=1,\dots, n$.
$$
{\mathcal M}(R, p_1,\dots, p_n):=\pi_0\bigl(\diff(R, p_1,\dots, p_n)\bigr).
$$
Up to isomorphism, 
this group depends only on the orientability and the genus of $R$.
The orientable case has been  intensely  studied.
Recent important results about the non-orientable case are in \cite{kor, wahl}.

(In the literature, $ {\mathcal M}_{g,n}$
is used to denote both the mapping class group
of an orientable genus $g$ (hence with Euler characteristic
$2-2g$) surface with $n$ marked points
and the mapping class group
of a non-orientable genus $g$ (hence with Euler characteristic
$2-g$) surface with $n$ marked points.)

\end{defn}

In preparation for the next section, we  establish a
somewhat new explicit set of generators in the
non-orientable case.

Write $R$ as
$B_{p_1,\dots,p_g}S^2$, the blow up of $S^2$ at $g$ points.
We start by describing some elements of 
the mapping class group. For more details
see \cite{lick, chi, kor}.

\begin{defn}[Dehn twist] Let $R$ be any surface and
$C\subset R$ a simple closed smooth curve such that
$R$ is orientable along $C$. Cut $R$ along $C$, rotate one side
around once completely and glue the pieces back together.
This defines a  diffeomorphism $t_C$ of $R$, see
Figure~\ref{fig.dehn.defn}.

\begin{figure}[ht]
    \centering
   \epsfysize=3.8cm
    \epsfbox{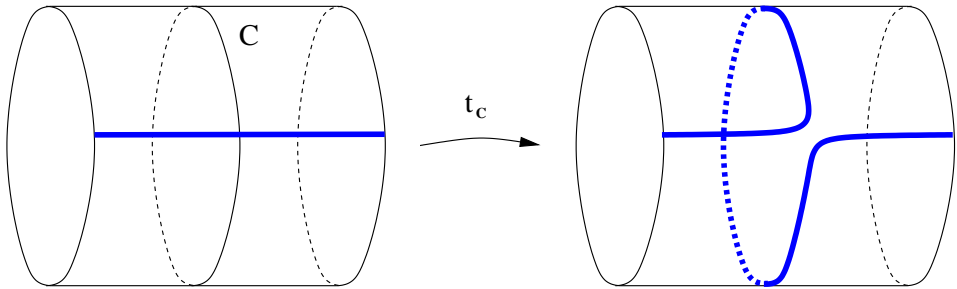} 
\caption{The effect of the Dehn twist around $C$ on a curve.}
    \label{fig.dehn.defn}
    
  \end{figure}
  
The inverse
$t^{-1}_C$ corresponds to rotating one side the other way.
Up to isotopy, the pair $\{t_C, t^{-1}_C\}$ does not depend on the 
choice of $C$ or the rotation.
 Either of $t_C$ and $ t^{-1}_C$ is called a {\it Dehn twist} using $C$.
On an oriented surface, with $C$ oriented, one can
make a sensible distinction
between $t_C$ and $ t^{-1}_C$. This is less useful in the
non-orientable case.
\end{defn}

\begin{defn}[Crosscap slide]\label{cc.slide.defn}
 Let $D$ be a closed disc
and $p,q\in D$ two points.  Take a simple closed curve $C$ in $D$ 
passing
through $p,q$ and let $C'$ denote the corresponding curve 
in $ B_qD$.
Let $M_p$ be a small disc around $p$.
Let $\{\phi_t:t\in [0,1]\}$ be a continuous
 family of diffeomorphisms of $B_qD$
such that $\phi_0$ is the identity, each $\phi_t$
 is  the identity near the boundary  and as $t$ increases, the $\phi_t$
slide   $M_p$  once around $C$. At $t=1$,  $M_p$ returns to itself
with its orientation reversed, as in
Figure~\ref{fig.cc.defn}. In particular, $\phi_1(p)=p$.
Thus $\phi_1$ can be lifted to a diffeomorphism of  $B_{p,q}D$ 
which is not isotopic to the identity but is the identity near the boundary.

\begin{figure}[ht]
    \centering
   \epsfysize=3cm
    \epsfbox{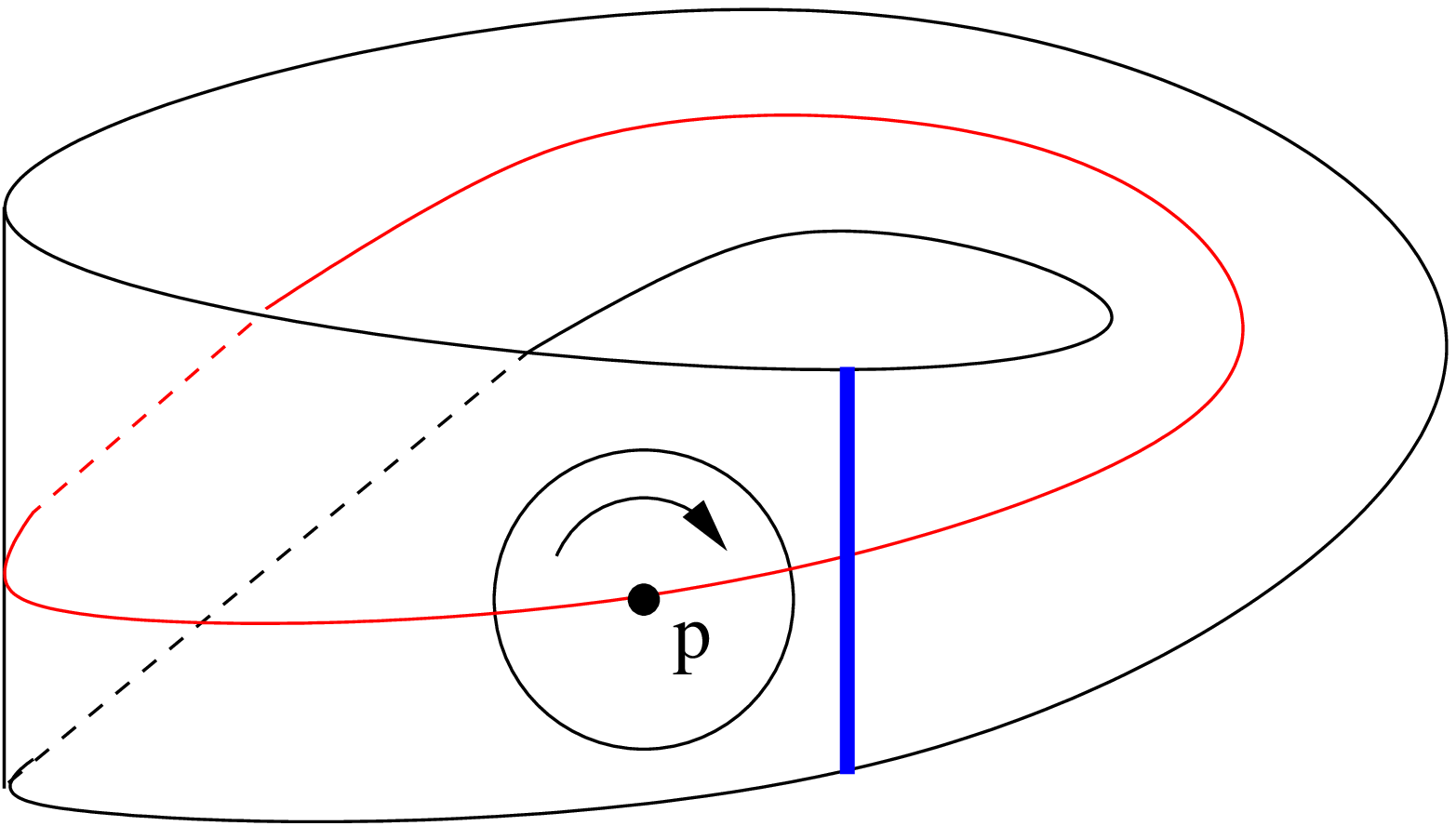}
\quad\raisebox{1.5cm}{ $\stackrel{\phi_1}{\longrightarrow}$}\qquad
 \epsfysize=3cm
    \epsfbox{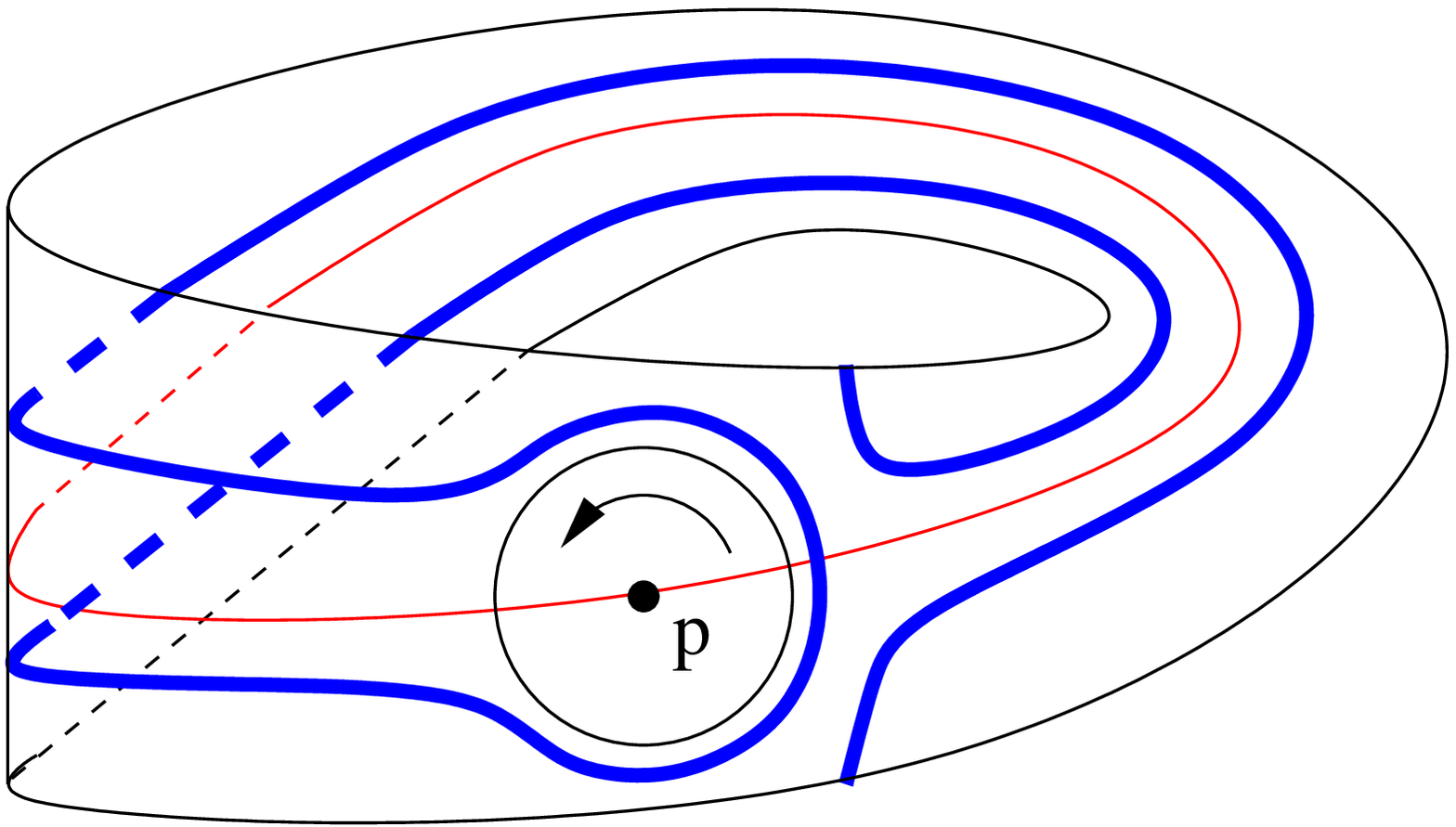}
\caption{Cross-cap slide.}
    \label{fig.cc.defn}
    
  \end{figure}
  
Let $R$ be any surface, $U\subset R$ a closed subset with $C^{\infty}$ boundary
and $\tau:U\to B_{p,q}D$ a  diffeomorphism.
Then $\tau^{-1}\phi_1\tau \colon U\to U$ is the identity near the boundary
of $U$, hence it can be extended by the identity
on $R\setminus U$ to a diffeomorphism of $R$.
 Up to isotopy, this diffeomorphism does not depend on the choice of $C$,
 $\phi_t$ and $\tau$. It is called a {\it cross-cap slide}
or a {\it $Y$-homeomorphism} using $U$. 
Note that for a cross-cap slide to exist, $R$ must be  non-orientable and of 
genus at least 2.
\end{defn}

\begin{say}[Generators of the mapping class group]\label{chill.gens}

Let $R_g$ be a non-orientable surface of genus $g\geq 1$.
We write $R_g:=B_{p_1,\dots,p_g}S^2$ with 
exceptional curves $E_i\subset R_g$
and let $\pi \colon R_g\to S^2$ be the blow down map.

The map $\pi$ gives  a one-to-one correspondence between
\begin{enumerate}
\item[$\bullet$] simple closed smooth curves $C_R\subset R_g$
 whose intersection with
any exceptional curve $E_i$ is transversal, and
\item[$\bullet$]  immersed curves $C=\pi(C_R)\subset S^2$
whose only self-intersections are at the $p_i$ and
no two branches are tangent.
\end{enumerate}

Generators of the mapping class group were first established by
\cite{lick} and simplified by \cite{chi}.
The case with marked points was settled by \cite{kor}.

The generators  are the following
\begin{enumerate}
\item Dehn twists along $C_R$ for certain smooth curves $C\subset S^2$ that
pass through an even number of the $p_1,\dots,p_g$.
(No self-intersections at the $p_i$.)
\item Cross-cap slides using a disc
$D\subset S^2$ that contains exactly 2 of the $p_1,\dots,p_g$.
\end{enumerate}

The results of \cite{chi} and of \cite{kor} are more precise
in that  only very few of these generators are used.
In the unmarked case, 
the above formulation is   established in the course of  the proof
and stated on \cite[p.427]{chi}.
\end{say}

We will need somewhat different generators.
We thank M.\  Korkmaz for answering many questions
and especially for pointing out that one should use 
 the lantern relation (\ref{lantern.defn}) to establish the
following. 

\begin{prop}\label{mapp.cl.our.gens}
 The following elements generate the mapping class group of
the marked surface
 $\bigl(B_{p_1,\dots,p_g}S^2, q_1,\dots, q_n)$.
\begin{enumerate}
\item Dehn twists along $C_R$ for certain smooth curves $C\subset S^2$ that
pass through $0,2$ or $4$ of the points $p_1,\dots,p_g$.
(No self-intersections at the $p_i$.)
\item Cross-cap slides using a disc
$D\subset S^2$ that contains exactly 2 of the points $p_1,\dots,p_g$.
\end{enumerate}
\end{prop}

Proof. We have included all the cross-cap slides from
(\ref{chill.gens}). Thus we need to deal with 
Dehn twists along $C_R$ where 
 $C\subset S^2$ is a simple closed  curve
passing through $m$ of the points $p_1,\dots, p_g$ with $m>4$.

 Using induction, it is enough to show that
the Dehn twist along $C_R$ can be written as the product
of Dehn twists along curves $C'_R$ where each $C'\subset S^2$ 
is a simple closed  curve
passing through fewer than $m$ of the points $p_1,\dots, p_g$.

Assume for simplicity that $C$ passes through 
$p_1,\dots, p_m$ with $m>4$ (and even).
For $I\subset \{1,\dots, m\}$
let $t_I$ be a Dehn twist using a simple closed curve $C_I$
passing through the $\{p_i:i\in I\}$ but none of the others.
The precise choice of the curve will be made later. 
We show that, with a suitable choice of the curves,
$t_{12345...m}$ is a product of the Dehn twists
$t_{125\dots m}, t_{345\dots m}, t_{1234}, t_{5\dots m}, t_{12}, t_{34}$.

This is best shown by a picture for $m=8$.
  In Figure~\ref{fig.lantern},
$t_{12345678}$ is a product of the Dehn twists
$t_{125678}, t_{345678}, t_{1234}, t_{5678}, t_{12}, t_{34}$. 
The shaded region is a sphere with four holes, and corresponds 
to a neighborhood of the lift to $R_8$ of $C_{12345678}$. 
On each side of the picture are drawn the curves corresponding to the
 Dehn twists of the same side in (\ref{lantern.defn}.1): 

a) $C_{12}$, $C_{34}$, $C_{5678}$, $C_{12345678}$,
 b) $C_{1234}$, $C_{125678}$, $C_{345678}$. \qed

 \begin{figure}[ht]
    \centering
     \epsfysize=4cm
    \epsfbox{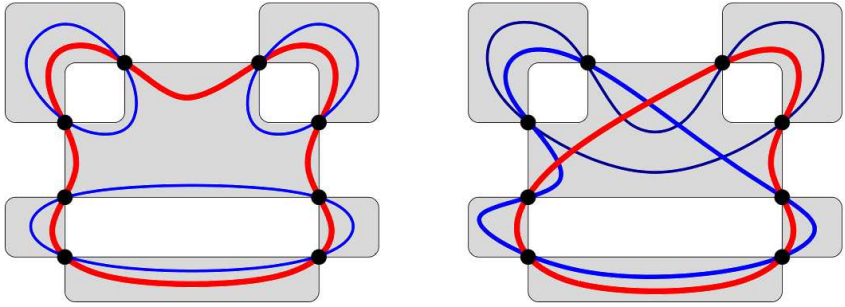}
    \caption{Lantern relation for $m=8$.}\label{fig.lantern}
     
  \end{figure}

\begin{say}[Lantern relation of Dehn]\label{lantern.defn}\cite{dehn,johnson}
 Fix 4 points 
$q_0,\dots, q_3\in S^2$. Let $t_i$ be the 
Dehn twist using a small circle around $q_i$ and for $i,j \in \{1,2,3\}$, let
$t_{ij}$  be the 
Dehn twist using a simple closed curve that separates
 $q_i, q_j$ from the other 2 points. Then, with suitable
orientations,
$$
t_0t_1t_2t_3 = t_{12}t_{13}t_{23}\;,
\eqno{(\ref{lantern.defn}.1)}
$$
where the equality is understood to hold in
${\mathcal M}(S^2, q_0,\dots, q_3)$.
\end{say}

\section{Automorphisms and the mapping class group}
\label{mapping.cl.proofs}

The main result of this section is the following.

\begin{thm} \label{map.class.gen.onto}
 Let $R$ be a real algebraic surface that is obtained from $S^2$
by blowing up points
and $p_1,\dots, p_n\in R$ distinct marked points. Then
the natural map
$$
\aut(R, p_1,\dots, p_n)\onto {\mathcal M}(R, p_1,\dots, p_n)
\qtq{is surjective.}
$$
\end{thm}

Proof. We prove that all the generators
of the mapping class group listed in (\ref{mapp.cl.our.gens})
 can be realized algebraically.
There are 4 cases to consider:
\begin{enumerate}
\item Dehn twists along $C_R\subset R$ for  smooth curves $C\subset S^2$ that
pass through either 
\begin{enumerate}
\item none of the points $p_i$,
\item exactly 2  of the points $p_i$, or
\item exactly 4  of the points $p_i$.
\end{enumerate}
\item Cross-cap slides using a disc
$D\subset S^2$ that contains exactly 2 of the points $p_i$.
\end{enumerate}

We start with the easiest case (\ref{map.class.gen.onto}.1.a).

\begin{say}[Algebraic realization of Dehn twists]\label{dehn.cb.conds}
Let  $C\subset S^2$ be a smooth curve passing through 
 none of the points $p_i$. After applying a  suitable
automorphism of $S^2$, we may assume that $C$ is the big circle
$(x=0)$.

Consider the map $g \colon [-1,1] \to O(2)$ where $g(t)={\mathbf 1}$ for
$t \in [-1,-\epsilon]\cup[\epsilon,1]$ and $g(t)$ is the rotation by 
angle $\pi(1 + t/\epsilon)$ for $t \in [-\epsilon,\epsilon]$. 
Let  $M \colon [-1,1] \to O(2)$ be an algebraic approximation
of $g$ such that the corresponding twisting (\ref{twist.defn})
$\Phi_M$ is the identity at the points $p_i$.
Then $\Phi_M$ is an algebraic realization of the Dehn twist
around $C$.

On the torus, the same argument works for
either of the $S^1$-factors. Up to isotopy and the
natural $GL(2,\z)$-action, this takes care of all
simple closed curves.
\end{say}

Next we deal with the hardest case (\ref{map.class.gen.onto}.1.c).

\begin{say}[4 pt case] After applying a  suitable
automorphism of $S^2$, we may assume that $C$ is close to a circle
in $S^2$ but  the 4 points do not lie on a circle. 

Let us take an annular neighborhood of $C$ and blow up the 4 points
$p_1,\dots, p_4$. The resulting open surface is denoted
by $W\subset B_{p_1,\dots, p_4}S^2$. It contains the curve $C_R$ and the
4 exceptional curves $E_1,\dots, E_4$. 

If we
cut the blown-up annulus $W$ along the 5 curves
$A_1,\dots,A_4,D$ as indicated of the left hand side
of  Figure~\ref{fig.annulus2},
we get the contractible surface  $U$ indicated on the right hand side
of  Figure~\ref{fig.annulus2}.
The left and right hand sides of $U$ are identified to form a cylinder,
giving a neighborhood of the curve $C_R\subset B_{p_1,\dots, p_4}S^2$.
The big rectangle with lighter shading in $U$ on the right
corresponds to the lighter shaded are in $W$ on the left.
The 4 top and 4 bottom line segments of $U$ are identified to form
4 M\"obius bands.

 \begin{figure}[ht]
    \centering
     \epsfysize=4.2cm
    \epsfbox{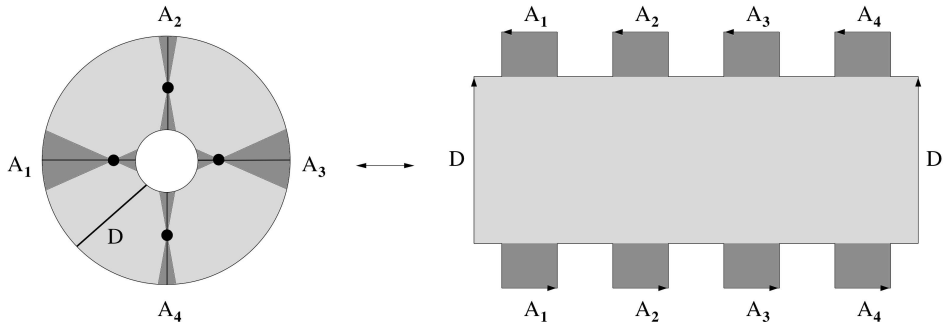}
    \caption{Two models of the annulus blown up in 4 points.}\label{fig.annulus2}
     
  \end{figure}

Next, in Figure~\ref{fig.dehn-k1} we show the 4 exceptional curves.

 \begin{figure}[ht]
    \centering
     \epsfysize=4cm
    \epsfbox{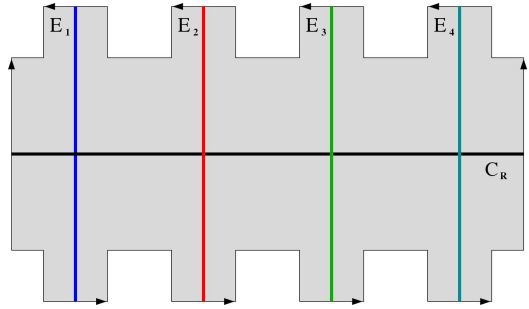}
    \caption{The 4 exceptional curves.}\label{fig.dehn-k1}
     
  \end{figure}

Figure~\ref{fig.dehn-k2}  shows the images of the curves $E_i$
after the Dehn twist around $C_R$. 

 \begin{figure}[ht]
    \centering
     \epsfysize=4cm
    \epsfbox{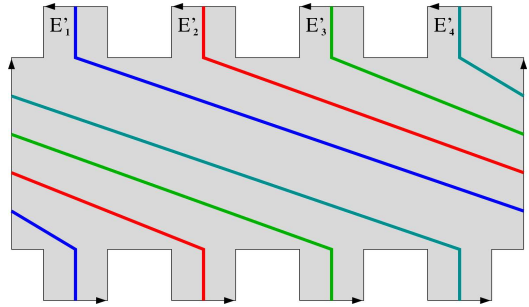}
    \caption{Effect of the Dehn twist around $C_R$.}\label{fig.dehn-k2}
     
  \end{figure}

These images can be deformed to obtain a configuration 
as in Figure~\ref{fig.dehn-k3}. 
Note that now $E_i$ intersects $E'_j$ iff $i\neq j$.

 \begin{figure}[ht]
    \centering
     \epsfysize=4cm
    \epsfbox{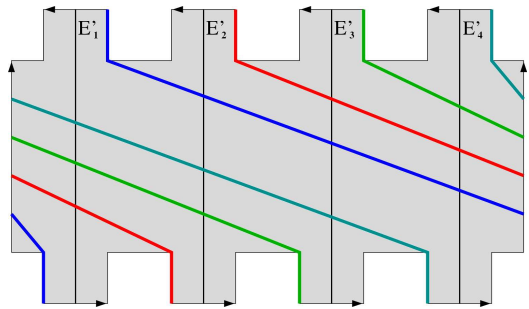}
    \caption{Deformation of Figure~\ref{fig.dehn-k2}.}\label{fig.dehn-k3}
     
  \end{figure}

Next we convert this back to the annulus model $W$
on the left hand side
of  Figure~\ref{fig.annulus2}. 

We obtain Figure~\ref{fig.dehn-g4d}.

 \begin{figure}[ht]
    \centering
      \epsfysize=6cm
    \epsfbox{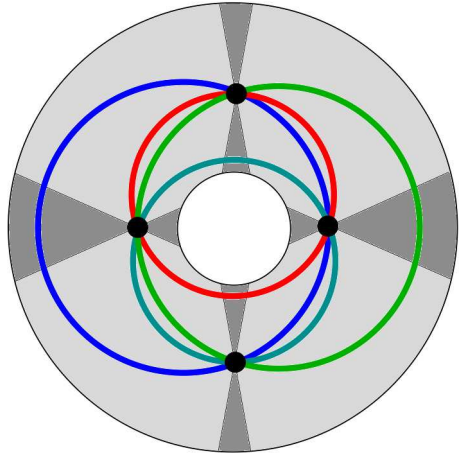} 
  \caption{Images of the four exceptional curves.}\label{fig.dehn-g4d}
  
       \end{figure}

The images of the exceptional curves $E_1,\dots, E_4$
under the standard Cremona transformation
with base points $p_1,\dots, p_4$ are shown in Figure~\ref{fig.cremona}. 

We see by direct inspection that the two quartets of curves
in Figures~\ref{fig.cremona} and \ref{fig.dehn-g4d} are isotopic.
Thus, if we first apply the Dehn twist and
then the (inverse) Cremona transformation and a suitable isotopy, 
we get a diffeomorphism
$\phi \colon R_n\to R_n$ such that $\phi(E_i)=E_i$. 
That is, $\phi$ is lifted from a diffeomorphism
of the g-pointed sphere $(S^2, p_1,\dots, p_n)$.
By (\ref{i.S2n.thm}), any such diffeomorphism is isotopic 
to an algebraic automorphism. Hence the 
Dehn twist along $C_R$ is also algebraic.
\end{say}

\begin{say}[2 pt case]\label{2pt.case.say}
 The proof is the same as in the 4 point case
but the description is easier.

A neighborhood of $C$ gives an annulus with 2 blown-up points.
After the Dehn twist we get two curves $E'_1, E'_2$ as in 
Figure~\ref{fig.cremona-2pts}.
\begin{figure}[ht]
   \epsfysize=3.5cm
    \epsfbox{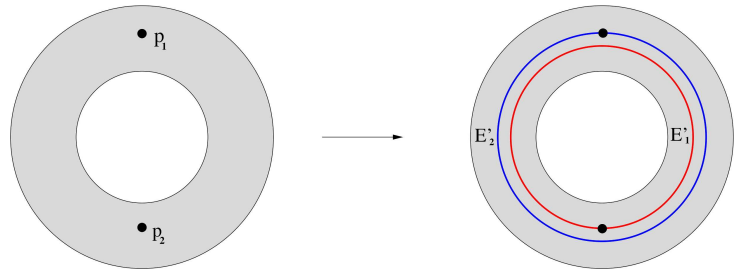}
 \caption{Cremona transformation with 2 real base points.}
        \label{fig.cremona-2pts}
        
      \end{figure}

We can assume that the two curves $E'_1, E'_2$ are close to being 
circles, that is, close to the intersections
 $S^2\cap H_i$ for some planes for $i=1,2$.
 Let $q,\bar q$ be the 2 (complex conjugate)
points where these 2 planes $H_i$ intersect the complexified sphere $Q$.
Then the Cremona transformation with base points
$p_1, p_2, q,\bar q$ is the inverse of the Dehn twist,
again up to  a diffeomorphism
of   $S^2$.
\end{say}

\begin{say}[Crosscap slides] Here the topological picture is
given by Figure~\ref{fig.crosscap}.
\begin{figure}[ht]
   \epsfysize=3.5cm
    \epsfbox{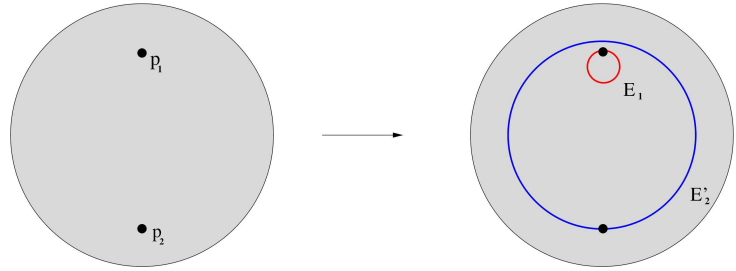}
 \caption{Cross-cap slides.}
        \label{fig.crosscap}
      \end{figure}
Note that $E_1$ is mapped to itself and $E_2$ is mapped 
to the (almost) circle $E'_2$. Up to isotopy, we can replace
$E_1$ with a small circle $E'_1$ passing through $p_1$. 

As in (\ref{2pt.case.say}), we obtain 
 $q,\bar q$ such that
 the Cremona transformation with base points
$p_1, p_2, q,\bar q$ is the inverse of the Dehn twist,
 up to  a diffeomorphism
of $S^2$.
\end{say}

\begin{say}[Proof of (\ref{homeo.gen.dense})]
Let $\phi:(R, q_1,\dots, q_n)\to  (R, q_1,\dots, q_n)$
be any diffeomorphism.
By (\ref{homeo.identity.dense}), there is an automorphism
$\Phi_1\in \aut(R, q_1,\dots, q_n)$ such that
$\Phi_1^{-1}\circ \phi$ is homotopic to the identity.

By (\ref{map.class.gen.onto}), we can approximate $\Phi_1^{-1}\circ \phi$ 
by a sequence of automorphisms $\Psi_s\in \aut(R, q_1,\dots, q_n)$.
Thus $\Phi_1\circ \Psi_s\in \aut(R, q_1,\dots, q_n)$
converges to $\phi$.
\qed
\end{say}

%%%%%%%%%%%%%%%%%%%%%%%%%%%%%%%%%%%%%%%%

\end{document}